\documentclass[reqno]{amsart}

\usepackage[author-year]{amsrefs}

\RequirePackage[OT1]{fontenc}
\RequirePackage{amsthm,amsmath}
\RequirePackage[colorlinks,citecolor=blue,urlcolor=blue,breaklinks=true]{hyperref}

\usepackage{amssymb}
\usepackage{verbatim}
\usepackage{graphicx}\graphicspath{{figures/}}
\usepackage{multicol}
\usepackage{tabularx}
\usepackage{mathrsfs} 
\usepackage{datetime}
\usepackage{dsfont}
\usepackage{bm}
\usepackage{enumitem}

\usepackage[capitalize]{cleveref}
\usepackage{mathtools}
\usepackage{thmtools}
\usepackage{color}

\def\[#1\]{\begin{align}#1\end{align}}
\newcommand{\defas}{\vcentcolon=}  
\newcommand{\given}{\mid}

\newcommand{\Reals}{\mathbb{R}}

\newcommand{\as}{\textrm{a.s.}}

\newcommand{\dee}{\mathrm{d}}

\DeclareMathOperator{\supp}{supp}

\DeclareMathOperator*{\newlim}{\mathrm{lim}\vphantom{\mathrm{infsup}}}

\DeclareMathOperator*{\newinf}{\mathrm{inf}\vphantom{\mathrm{infsup}}}

\renewcommand{\lim}{\newlim}

\renewcommand{\inf}{\newinf}

\newcommand{\Nats}{\mathbb{N}}
\newcommand{\Ints}{\mathbb{Z}}
\newcommand{\NNReals}{\Reals_+}
\newcommand{\NNInts}{\Ints_+}

\renewcommand{\Pr}{\mathbb{P}}
\newcommand{\defn}[1]{\emph{#1}}

\newcommand{\ceiling}[1]{\lceil #1 \rceil}

\def\EE{\mathbb{E}}

\newcommand{\Uniform}{\textrm U (0,1)}

\newcommand{\dist}{\ \sim\ }

\newcommand{\bspace}{\Omega}
\newcommand{\bsa}{\mathcal A}
\newcommand{\borelspace}{(\bspace,\bsa)}

\newcommand{\BPLAW}{\mathrm{BP}}

\newcommand{\distind}{\overset{ind}{\dist}}
\newcommand{\iid}{i.i.d.}
\newcommand{\pmf}{p.m.f.}

\newcommand{\nprocess}[3]{(#1_{#3})_{#3 \in #2}}
\newcommand{\process}[2]{\nprocess{#1}{#2}n}

\newcommand{\Atoms}{\mathscr{A}}
\newcommand{\theset}[1]{\lbrace #1 \rbrace}

\newcommand{\BM}{B_0} 
\newcommand{\NABM}{\tilde B_0}

\DeclareMathOperator{\betadist}{beta}

\DeclareMathOperator{\geodist}{geometric}
\DeclareMathOperator{\berndist}{Bernoulli}
\DeclareMathOperator{\nbdist}{NB}

\DeclareMathOperator{\nbfactory}{NB-factory}

\DeclareMathOperator{\BEPLAW}{BeP}
\DeclareMathOperator{\NBPLAW}{NBP}

\newcommand{\bb}{\bar b} 

\newcommand{\Fcal}{\mathcal{F}}

\newcommand{\gprocess}[2]{(#1)_{#2}}

\newcommand{\fixedvar}{\vartheta}
\newcommand{\arbmeas}{\xi}

\newcommand{\RBM}{B}
\newcommand{\Yscr}{\mathscr Y}
\newcommand{\Hscr}{\mathscr H}
\newcommand{\NRBM}{B_0}

\newcommand{\BMspace}{\mathcal M_0 \borelspace}

\numberwithin{equation}{section}
\theoremstyle{plain}
\newtheorem{thm}{Theorem}[section]
\newtheorem{lem}{Lemma}[section]
\newtheorem{prop}{Proposition}[section]
\newtheorem{definition}{Definition}[section]
\newtheorem{algorithm}{Algorithm}[section]

\begin{document}

\title[Black-box multisets]{Black-box constructions for exchangeable sequences of random multisets}

\author{Creighton Heaukulani}
\address{University of Cambridge\\
Cambridge, UK}
\email{c.k.heaukulani@gmail.com} 
\thanks{CH was supported by the Stephen Thomas studentship at Queens' College, Cambridge, with funding also from the Cambridge Trusts.}

 \author{Daniel M. Roy}
 \address{University of Toronto\\
 Toronto, Canada}
 \email{droy@utstat.toronto.edu}
 \thanks{DMR carried out this research while a research fellow of Emmanuel College, Cambridge, with funding also from a Newton International Fellowship through the Royal Society.}

\date{}

\begin{abstract}
We develop constructions for exchangeable sequences of point processes that are rendered conditionally-\iid\ negative binomial processes by a (possibly unknown) random measure called the \emph{base measure}.
Negative binomial processes are useful in Bayesian nonparametrics as models for random multisets, and in applications we are often interested in cases when the base measure itself is difficult to construct (for example when it has countably infinite support).
While a finitary construction for an important case (corresponding to a beta process base measure) has appeared in the literature, our constructions generalize to \emph{any} random base measure, requiring only an exchangeable sequence of Bernoulli processes rendered conditionally-\iid\ by the same underlying random base measure.
Because finitary constructions for such Bernoulli processes are known for several different classes of random base measures -- including generalizations of the beta process and hierarchies thereof -- our results immediately provide constructions for negative binomial processes with a random base measure from any member of these classes.
\end{abstract}

\maketitle


\section{Introduction}

A multiset is a set with possible repetitions of its elements.
A popular class of models for \emph{random multisets} in Bayesian nonparametric applications are the \defn{negative binomial processes}, which have been applied as topic models in document analysis and as latent factor models for image segmentation and object detection in computer vision, among other applications \cites{HR2014,BMPJpre,ZHDC12}.
In this article, we study exchangeable sequences $\process X \Nats \defas (X_1, X_2, \dots)$ of point processes on a measurable space that are rendered conditionally-\iid\ by a random measure $B$ called the \emph{base measure}.
Borrowing language from the theory of exchangeable sequences, we say that $B$ \defn{directs} the exchangeable sequence $\process X \Nats$.
Unconditionally, the measures $\process X \Nats$ will in general not be negative binomial processes, and we therefore refer to $\process X \Nats$ as an \emph{exchangeable sequence of multisets directed by $B$}.

In this work, we present algorithms to construct $\process X \Nats$ from any exchangeable sequence $\process Y \Nats \defas (Y_1, Y_2, \dots)$ of \emph{Bernoulli processes} directed by $B$.
(We review Bernoulli processes in \cref{sec:prelims}.)
So long as the total mass of $B$ is almost surely (\as) finite,
our constructions are also \emph{finitary}. 
That is, even if the support of $B$ is \as\ infinite, our construction of each $X_n$ is, with probability one, entirely determined by the finite set of atoms in the support of some prefix of $\process Y \Nats$.
In particular, our construction makes no direct use of $B$ and so $B$ need not even be represented explicitly.
Such constructions are useful for several reasons, notably:
\begin{enumerate}
\item In Bayesian nonparametric applications, we are interested in cases when $B$ has a countably infinite set of atoms, e.g., in the popular \emph{beta process} \cites{Hjort1990,TJ2007}.

\item Different models may be imposed on the base measure $B$ for various applications, e.g., generalizations of the beta process \cites{TG2009,Roy13CUP,HR14Gibbs} and hierarchies thereof \cites{TJ2007,Roy13CUP}, in which case it is convenient to have a black-box method.
\end{enumerate}
For the case when $B$ is a beta process (a precise definition is given in \cref{sec:mainresult}), a finitary construction for $\process X \Nats$ was given by \ocite{HR2014} (as well as by \ocite{ZMPS2016} for a reparameterization of the beta process), which takes advantage of conjugacy between beta processes and negative binomial processes \cites{BMPJpre,ZHDC12,Hjort1990,Kim1999}.
However, this approach does not generalize easily to other classes of base measures.
Therefore, instead of tailoring constructions to different cases, our approach provides a black-box method to construct $\process X \Nats$, assuming only that we have access to some exchangeable sequence $\process Y \Nats$ of Bernoulli processes directed by $B$.

Finitary constructions for exchangeable sequences $\process Y \Nats$ of Bernoulli processes 
are known for several different classes of directing random base measures.
For example, when $B$ is a beta process, one finitary construction for $\process Y \Nats$ is provided by the \defn{Indian buffet process} (IBP) \cites{GG06,GGS2007}.
\ocite{TG2009} generalized the beta process to the \defn{stable beta process} and provided a finitary construction for $\process Y \Nats$ in this case by generalizing the IBP to the \defn{stable IBP} (studied further by \ocite{BJP2012}), which was shown to exhibit power-law behavior in latent feature modeling applications.
\ocite{Roy13CUP} provided a further generalization to a large class of random base measures called \defn{generalized beta processes}, along with a corresponding generalization of the IBP.
A special subclass called \emph{Gibbs-type beta processes} (corresponding to a \emph{Gibbs-type IBP}) was studied by \ocite{HR14Gibbs}, which broadened the profile of attainable power-law behaviors beyond those achieved with the stable IBP.

Another useful modeling paradigm is obtained by organizing random base measures into hierarchies (see \ocite{TJ2007} for the prototypical example).
Such random base measures are useful in admixture or mixed-membership models, where there is latent structure shared between several distinct groups of data.
\ocite{Roy13CUP} provided a finitary construction for $\process Y \Nats$ directed by a hierarchy of generalized beta processes, which, as discussed, includes hierarchies of all previously mentioned random base measures as special cases.
In \cref{sec:mainresult}, we will illustrate the application of our construction when the directing random measure is a hierarchy of beta processes.

Finally, we note that alternative methods to construct $\process X \Nats$ directed by random base measures with a countably infinite number of atoms may be obtained with \emph{stick-breaking constructions} \cites{TGG07,PZWGC2010} or \emph{inverse L\'evy measure methods} \cite{WI1998}.
These constructions truncate the number of atoms in the underlying base measure and are therefore not exact, so Markov chain Monte Carlo (MCMC) techniques need to be introduced in order to remove this error, as in \ocites{BMPJpre,ZHDC12}.
Again, these approaches must be tailored to each specific case and are only accessible if such alternative representations for the random base measure exist.
Moreover, the representation is only exact in the asymptotic regime of the Markov chain.
Our approach is to instead avoid representing the underlying random base measure altogether, which has practical benefits (in addition to it being a black-box method) as MCMC subroutines need not be implemented for the simulation of $\process X \Nats$.

The remainder of the article is organized as follows.
We provide background and formally define notation in \cref{sec:prelims}.
In \cref{sec:mainresult}, we present our black-box construction in the case when the parameter $r$ (of the law of the negative binomial process) is an integer, which takes an intuitive approach.
We conclude in \cref{sec:fractional} by applying a rejection sampling subroutine in order to generalize our constructions to any parameter $r>0$.

\section{Notation and background}
\label{sec:prelims}

%
The focus of this article is on exchangeable sequences of random multisets and their de Finetti (mixing) measures.
Let $\bspace$ be a complete, separable metric space equipped with its Borel $\sigma$-algebra $\bsa$ and let $\NNInts \defas \{0, 1, 2, \dotsc \}$ denote the non-negative integers.  We represent multisets of $\bspace$ by $\NNInts$-valued random measures. In particular, by a \defn{point process}, we will mean a random measure $X$ on $\borelspace$ such that $X(A)$ is a $\NNInts$-valued random variable for every $A \in \bsa$.
Because $\borelspace$ is Borel, we may write $X = \sum_{k\le \kappa}\delta_{\gamma_k}$ for some random elements $\kappa$ in $\NNInts \cup \{\infty\}$ and (not necessarily distinct) $\gamma_1, \gamma_2, \dotsc$ in $\bspace$.
We will take $X$ to represent the \emph{multiset} of its unique elements $\gamma_k$ with corresponding multiplicities $X \theset{\gamma_k}$.
%

\subsection{Completely random measures}

We build on the theory of \emph{completely random measures} \cite{Kallenberg2002}*{Ch.~12}; \cite{Kingman1967}. 
Recall that every completely random measure $\arbmeas$ can be written as a sum of three independent parts
\[
\arbmeas = \tilde \arbmeas +
	\sum_{s\in\Atoms} \fixedvar_s \delta_s
	+ \sum_{(s,p) \in \eta} p\, \delta_s 
	\qquad \as,
	\label{eq:general_CRM}
\]
called the \defn{diffuse}, \defn{fixed}, and \defn{ordinary} components, respectively, where:
\begin{enumerate}
\item $\tilde \arbmeas$ is a non-random, non-atomic measure;

\item $\Atoms \subseteq \bspace$ is a non-random countable set whose elements are referred to as the 
\defn{fixed atoms} and whose masses $\fixedvar_1, \fixedvar_2, \dotsc$ are independent random variables in $\NNReals$ (the non-negative real numbers); 

\item $\eta$ is a Poisson process on $\bspace \times (0,\infty)$ whose intensity measure $\EE \eta$ is $\sigma$-finite and has diffuse projections onto $\bspace$, i.e., the measure $(\EE \eta) (\,\cdot\, \times (0,\infty))$ on $\bspace$ is non-atomic.
\end{enumerate}
%

\subsection{Base measures}

Let $\BMspace$ denote the space of $\sigma$-finite measures on $\borelspace$ whose atoms have measure less than one\footnote{Equipped with the $\sigma$-algebra generated by the projection maps $\mu \mapsto \mu(A)$, for all $A\in\bsa$.}.
Elements in $\BMspace$ are called \emph{base measures}.
For the remainder of the article, fix a base measure $\BM$ in $\BMspace$ given by
\[ 
\label{eq:base_meas}
\BM = \NABM + \sum_{s\in\Atoms} \bb_s \delta_s
\]
for some non-atomic measure $\NABM$;
a countable set $\Atoms \subseteq \bspace$;
and constants $\bb_1, \bb_2, \dotsc$ in $(0,1]$.
%

\subsection{Negative binomial processes}

We say that a random variable $Z$ in $\NNInts$ has a \defn{negative binomial distribution with parameters $r>0$ and $p\in(0,1)$}, written $Z \dist \nbdist(r,p)$, if its probability mass function (\pmf) is given by
\[
\Pr \{Z = k\} = \frac{ (r)_k }{ k! } p^k (1-p)^r
	,
	\qquad k = 0, 1, \dotsc,
\]
where $(a)_n \defas a(a+1)\cdots (a+n-1) = \Gamma(a + n) / \Gamma(a)$ denotes the $n$-th rising factorial (and its analytic continuation).
\begin{definition}[negative binomial process]
We call a point process $X$ on $\borelspace$ a \defn{negative binomial process with parameter $r>0$ and base measure $\BM$}, written $X \dist \NBPLAW(r, \BM)$, if it is purely atomic and completely random with fixed component
\[
\sum_{s\in\Atoms} \fixedvar_s \delta_s , 
        \qquad \fixedvar_s \distind \nbdist ( r, \bb_s )
        ;
\]
and with an ordinary component that has intensity measure
\[
(\dee s, \dee p) \mapsto r\, \delta_1(\dee p)\, \NABM(\dee s) .
\]
\end{definition}
The fixed component of this process was originally defined in \ocites{BMPJpre,ZHDC12}, and by \ocite{ThibauxThesis} for the case when $r=1$, corresponding to a \defn{geometric process}.
The ordinary component was additionally specified in \ocite{HR2014}, which we note is simply a Poisson (point) process on $\bspace$ with intensity measure $r\NABM$, and in \cref{sec:mainresult} we will see that this specification is natural.

We may alternatively characterize the law of a negative binomial process with its Laplace functional; the following may be verified with an application of the L\'evy-Khinchin theorem (see \ocite{HR2014}*{Sec.~2.2}).
\begin{prop}
Let $X \dist \NBPLAW(r, \BM)$. The Laplace functional of the law of $X$ is given by
\[
\EE [ e^{-X(f)} ]
	= 
	\exp \left [ 
			- \int \left ( 1 - e^{- f(s)} \right )  r \NABM ( \dee s ) 
		\right ]
	\prod_{s\in\Atoms} \biggl [ 
		\frac{ 1- \BM\theset s } { 1 - \BM\theset s e^{-f(s)} } 
		\biggr ]^r
		,
		\label{result:cf_nbp}
\]
for every measurable function $f \colon \bspace \to \NNReals$, where $X (f) \defas \int f(x) X (\dee x)$. 
\end{prop}

\subsection{Bernoulli processes}

As mentioned in the introduction, our algorithms require an exchangeable sequence of \emph{Bernoulli processes}, a class of completely random measures defined in this context by \ocites{Hjort1990,TJ2007}, though it should not be confused with the classic Bernoulli process studied in statistics and probability.
\begin{definition}[Bernoulli process]
We call a point process $X$ on $\borelspace$ a \defn{Bernoulli process with base measure $\BM$}, written $X\dist \BEPLAW(\BM)$, if it is purely atomic and completely random with fixed component
\[
\sum_{s\in\Atoms} \fixedvar_s \delta_s ,
	\qquad \fixedvar_s \distind \berndist(\bb_s)
	;
\]
and with an ordinary component that has intensity measure
\[
(\dee s, \dee p) \mapsto \delta_1(\dee p)\, \NABM(\dee s)
	.
\]
\end{definition}
Note that the ordinary component here is a Poisson process on $\bspace$ with intensity measure $\NABM$.
Also note that the Bernoulli process is \as\ \emph{simple} (i.e., has unit-valued atomic masses) and finite.
%

\subsection{Summary}

We now summarize the article more formally: our results provide an algorithm parameterized by some $r>0$ and takes as input an exchangeable sequence $\process Y \Nats$ of simple point processes on $\borelspace$, and outputs a sequence $\process X \Nats$ of point processes on $\borelspace$.
If $\process Y \Nats$ satisfies
\[
Y_n \given B \dist \BEPLAW(B)
	,
	\qquad n\in\Nats
	,
\]
for some random element $B$ in $\BMspace$, then $\process X \Nats$ satisfies 
\[
X_n \given B \dist \NBPLAW(r,B)
	,
	\qquad n \in \Nats
	.
\]
Importantly, we need not explicitly represent $B$, and note in particular that other than being $\sigma$-finite, the random base measure $B$ may be arbitrary: it need not be purely atomic nor completely random.

\section{A negative binomial urn scheme}
\label{sec:mainresult}

We now present our main construction for the case when $r$ is a positive integer, followed by a few demonstrative examples.
Let $r\in \Nats$ and let $\gprocess{Y_{n,m}}{n,m \in\Nats}$ be an array of simple point processes on $\borelspace$.
For every $n\in\Nats$,
\begin{enumerate}[leftmargin=*]
\item Define
\[
\Yscr_n 
	\defas \bigcup_{\ell\le r} \supp(Y_{n,\ell})
	,
	\label{eq:NBUS_sets}
\]
where $\supp(A)$ denotes the support of $A$;

\item We may write $\Yscr_n = \theset{ \gamma_{n,1}, \dotsc, \gamma_{n,\kappa_n} }$
for some random element $\kappa_n$ in $\NNInts$ and \as\ unique random elements $\gamma_{n,1}, \gamma_{n,2}, \dotsc$ in $\bspace$;

\item Define a sequence $\process X \Nats$ of point processes on $\borelspace$ where,
for every $j\in\Nats$,
\[
X_n \theset {\gamma_{n,j}} 
	\defas 	
		\inf \Bigl \{ 
		m \in \NNInts \colon m = \sum_{\ell=1}^{m+r} Y_{n,\ell}\theset {\gamma_{n,j} } 
		\Bigr \}
		,
		\;
		\text{on the event } \{ j \le \kappa_n \}
		,
	\label{eq:NBUS_construct}
\]
and $X_n \theset {\gamma_{n,j}} \defas 0$ otherwise;

\item For every ${A\in\bsa}$, put
$
X_n (A) \defas \sum_{s\in A \cap \Yscr_n} X_n \theset s.
$
These definitions imply that the measure $X_n$ is \as\ concentrated on a subset of $\Yscr_n$, i.e., ${X_n ( \bspace \setminus \Yscr_n ) = 0}$ \as
\end{enumerate}

\begin{definition}[negative binomial urn scheme] \label{defn:NBUS}
We call $\process X \Nats$ a \defn{negative binomial urn scheme induced by $\gprocess{Y_{n,m}}{n,m \in\Nats}$ with parameter $r$}.
\end{definition}
For intuition, if we think of $Y_{n,1} \theset \gamma, Y_{n,2} \theset \gamma, \dots$ in \cref{eq:NBUS_construct} as a sequence of independent Bernoulli trials each with unknown success probability, then $X_n\theset \gamma$ simply counts the number of successes in the sequence before $r$ failures, i.e., it has a negative binomial distribution.
The construction of negative binomial variates from Bernoulli variates is central to the article (see \cref{result:nb_construct} for a precise algorithm) and the following result may be thought of as an infinite dimensional analogue of this construction.

\begin{thm} \label{result:intro_arbH}
Let $r\in\Nats$, let $B$ be a random element in $\BMspace$, and let $\gprocess{Y_{n,m}}{n,m \in\Nats}$ be an exchangeable array of Bernoulli processes directed by $B$.
Let $\process X \Nats$ be a negative binomial urn scheme induced by $\gprocess{Y_{n,m}}{n,m \in\Nats}$ with parameter $r$.
Then, conditioned on $B$, the $\process X \Nats$ are \iid\ negative binomial processes with parameter $r$ and base measure $B$.
\end{thm} 

\begin{proof} 
It is clear that the random measures $X_1, X_2,\dotsc$ are conditionally independent given $\RBM$.  Fix $n\in\Nats$.  We must show that $X_n \given \RBM \dist \NBPLAW(r,\RBM)$.
Let $f \colon \bspace \to \NNReals$ be a measurable function 
and recall the notation $X (f) \defas \int f(x) X(\dee x)$.
We have
\[
   \EE[ \exp( -X_n (f) ) \given \RBM]  = g(\RBM)
   \label{eq:gdef}
\]
for some measurable function $g \colon \BMspace \to \NNReals$. 
Let $Y_n \defas \gprocess {Y_{n,m}}{m \in \Nats}$, and note that $X_n$ is $\sigma(Y_n)$-measurable, so there exists a measurable function $h$ such that $h(Y_n) = \exp ( - X_n (f) )$ \as. 
We have $Y_n \given B \dist ( \BEPLAW(B) )^\infty$, where the right-hand side denotes the infinite dimensional product measure, 
and so by the disintegration theorem \cite{Kallenberg2002}*{Thm.~6.4} we have
\[
  g(\RBM) = \int h \, \dee ( \BEPLAW (\RBM) )^\infty 
  	,
	\quad \as
	\label{eq:disintegration}
\]
We may therefore characterize $g$ using the structure of $h$ (i.e., without regard to the random base measure $B$), and if we show that it has the form of the Laplace functional of the (law of the) negative binomial process, then \cref{eq:disintegration} extends the result to the randomization in \cref{eq:gdef}.

Let $\RBM = \NRBM$ \as\ for some non-random measure $\NRBM \in \BMspace$ whose set of atoms we denote by $\Hscr_0$. 
We have that $\gprocess{Y_{n,m}\theset s}{m \in \Nats, s\in\Hscr_0}$ are independent random variables and 
\[
Y_{n,m} \theset s \dist \berndist ( \NRBM\theset s ) ,
	\qquad m \in \Nats, \ s\in\Hscr_0 .
\]
Then $\gprocess{X_n\theset s}{s\in\Hscr_0}$ are independent random variables constructed as in \cref{eq:NBUS_construct}, and by \cref{result:nb_construct},
\[
X_n\theset s \dist \nbdist( r, \NRBM \theset s ) ,
	\qquad s\in \Hscr_0
	.
	\label{eq:nbfixed}
\]

Let $\tilde Y_{n,m} \defas Y_{n,m}(\, \cdot \setminus \Hscr_0)$, for every $m\in\Nats$, be the restrictions of the Bernoulli processes to their ordinary components (which we recall are independent from the fixed components).
Then the $\gprocess{\tilde Y_{n,m}}{m\in\Nats}$ are independent Poisson processes (independent also from $\gprocess {X_n \theset s}{s \in \Hscr_0}$), each with intensity measure $\tilde \NRBM \defas \NRBM( \, \cdot \setminus \Hscr_0)$.
It follows that $\bigcap_{m\in\Nats} \supp(\tilde Y_{n,m}) = \emptyset$ \as\
Next, let $\tilde X_n \defas X_n ( \, \cdot \setminus \Hscr_0)$; 
it is straightforward to verify that $\tilde X_n = \sum_{j \le r} \tilde Y_{n,j}$ \as, and so $\tilde X_n$ is a Poisson process with intensity measure $r \tilde \NRBM$.
Then
\[
\begin{split} \label{eq:cf_Xn}
g(\NRBM) &= \EE [ e^{-X_n (f)} ]
	\\
	&= \EE \Bigl [ e^{- \tilde X_n (f) } \Bigr ] \times \EE \Bigl [ e^{- (X_n - \tilde X_n) (f) } \Bigr ]
	\\
	&= 
	\exp \Bigl [
		- \int ( 1 - e^{-f(s)} ) r \tilde \NRBM (\dee s)
	\Bigr ]
	 \times
  	\prod_{s\in \Hscr_0} \Bigl [
		\frac{ 1-\NRBM\theset s }{ 1-\NRBM\theset s e^{-f(s)} }
		\Bigr ]^r
	,
\end{split}
\]
where the factors in the second term of the last line are obtained from the Laplace transform of the negative binomial distribution.
This is the Laplace functional of the law of the negative binomial process with parameter $r$ and base measure $\NRBM$, as desired.
\end{proof}

We now provide two illustrative examples: when the directing measure is (1) the beta process \cite{Hjort1990} and (2) a hierarchy of beta processes \cite{TJ2007}.
Both of these random base measures are purely atomic and completely random, however, we note that in general the directing measure in \cref{result:intro_arbH} may have a diffuse component or may not even be completely random.

For the remainder of the section, let $c\colon \bspace \to \NNReals$ be a non-negative measurable function, which we call a \emph{concentration function}.
\begin{definition}[beta process]
We call a random base measure $B$ in $\BMspace$ a \defn{beta process with concentration function $c$ and base measure $\BM$} and we write $B\dist \BPLAW(c,\BM)$, if it is purely atomic and completely random with fixed component
\[
\sum_{s\in\Atoms} \fixedvar_s \delta_s ,
	\qquad \fixedvar_s \distind \betadist( c(s) \bb_s , c(s) ( 1- \bb_s ) )
	,
\]
and with an ordinary component that has intensity measure
\[
(\dee s, \dee p) \mapsto c(s) p^{-1} (1-p)^{c(s)-1} \NABM ( \dee s )
	.
	\label{eq:bpord}
\]
\end{definition}

Because the measure in \cref{eq:bpord} is not finite, the beta process has an infinite number of atoms \as\
However, consider the following construction:
Let $\process Y \Nats$ be a sequence of simple point processes on $\borelspace$, where ${Y_1 \dist \BEPLAW(\NABM)}$ and 
\[
Y_{n+1} \given Y_{[n]}
	\dist \BEPLAW \Bigl (
		\frac{c}{c+n} \NABM + \frac 1 {c+n} \sum_{j=1}^n Y_j
		\Bigr )
		,
		\qquad
		n\ge 1
		,
		\label{eq:IBP}
\]
where $Y_{[n]} \defas (Y_1, \dotsc, Y_n)$.
\ocite{TJ2007} showed that $\process Y \Nats$ is an exchangeable sequence of Bernoulli processes directed by a beta process $B\dist \BPLAW(c,\NABM)$, and that the \emph{combinatorial structure} of this sequence is in a sense described by the \emph{Indian buffet process} \cites{GG06,GGS2007}, which has found many uses in latent feature modeling applications \cite{GG2011}.
Passing the sequence $\process Y \Nats$ into the construction in \cref{result:intro_arbH}, we would obtain an exchangeable sequence of multisets directed by $B$, which is an alternative to the construction already provided for this case by \ocite{HR2014}.

Hierarchies of random base measures have also found many uses in Bayesian nonparametrics as admixture or mixed-membership models \cite{TJ2007}.
In particular, we call a random base measure $H$ in $\BMspace$ a \defn{hierarchy} of beta processes if there exists a beta process $B \dist \BPLAW(c,\BM)$ such that
\[
H \given B \dist \BPLAW(c, B)
	.
\]
\ocite{Roy13CUP} provides the following construction for an exchangeable sequence of Bernoulli processes directed by $H$, which takes as input the exchangeable sequence of Bernoulli process $\process Y \Nats$ directed by $B$ defined in \cref{eq:IBP}:
Let $\process W \Nats$ be a sequence of simple point processes on $\borelspace$ with $W_1 = Y_1$ \as\ and 
\[
W_{n+1} \given W_{[n]}, \ Y_{n+1}
        \dist \BEPLAW \biggl ( 
                \frac{c}{c + n} Y_{n+1} + \frac{1}{c + n} \sum_{j=1}^n W_j
        \biggr )
        ,
        \qquad
        n\ge 1.
\]

\begin{prop}[one-parameter process; \ocite{Roy13CUP}] \label{result:OPUS}
There exists an \as\ unique random element $H$ in $\BMspace$ such that, conditioned on $B$, $H$ is a beta process with concentration function $c$ and base measure $B$.
Furthermore, conditioned on $H$, the $\process W \Nats$ are \iid\ Bernoulli processes with base measure $H$.
\end{prop}

Roy calls $\process W \Nats$ a \emph{one-parameter process induced by $\process Y \Nats$ with concentration function $c$}, and comparing \cref{result:OPUS} to \cref{result:intro_arbH}, we can think of the negative binomial urn scheme as a negative binomial extension of the one-parameter process.
The following construction for an exchangeable sequence of negative binomial processes directed by a hierarchy of beta processes follows straightforwardly from \cref{result:intro_arbH} and \cref{result:OPUS}:
\begin{thm} \label{result:nbusexample}
Let $\process W \Nats$ and $H$ be as in~\cref{result:OPUS},
and arbitrarily arrange $\process W \Nats$ into an array $\gprocess{W_{n,m}}{n,m\in\Nats}$.
Let $\process X \Nats$ be a negative binomial urn scheme induced by $\gprocess{W_{n,m}}{n,m\in \Nats}$ with parameter $r\in\Nats$.
Then, conditioned on $H$, the $\process X \Nats$ are \iid\ negative binomial processes with parameter $r$ and base measure $H$.
\end{thm}

It is straightforward to see that the one-parameter process can be repeatedly applied to produce an exchangeable sequence of Bernoulli processes (and thus negative binomial processes) directed by an arbitrarily deep hierarchy of beta processes.

As discussed in the introduction, \ocite{Roy13CUP} also defined a generalization of the beta process with a broad class of random base measures called \defn{generalized beta processes}, which contains the beta process, the \defn{stable beta process} \cites{TG2009,BJP2012}, and the \emph{Gibbs-type beta process} \cite{HR14Gibbs} as special cases.
A finitary construction for an exchangeable sequence of Bernoulli processes directed by a generalized beta process (as well as its hierarchies) is provided therein.  
Passing such a sequence through the negative binomial urn scheme as in \cref{result:nbusexample} immediately yields an exchangeable sequence of negative binomial processes directed by any member from among these subclasses, e.g., stable beta processes, hierarchies of stable beta processes, hierarchies of Gibbs-type beta processes, etc.

\section{Generalization to $r>0$}
\label{sec:fractional}

The negative binomial urn scheme in \cref{result:intro_arbH} is only valid for positive integer values of $r$, and we now generalize this construction to any positive parameter value $r>0$.
Recall that we require a method to simulate a negative binomial variate $X \dist \nbdist(r,p)$ given only an \iid\ sequence of \emph{$p$-coins}, i.e., an \iid\ sequence of Bernoulli variates $Z_1, Z_2, \dotsc \dist \berndist(p)$, where $p$ is unknown.
For integer $r$, we simply recorded the number of heads before $r$ tails in the sequence, accomplished by \cref{eq:NBUS_construct}.
However, for non-integer $r>0$, we require a different approach.
In the following algorithm, we propose a $\nbdist(\ceiling{r},p)$ variate (simulated with the $p$-coins), and accept or reject the proposal with a simple rejection sampler.
Recall that $(a)_n \defas a(a+1)\cdots (a+n-1) = \Gamma(a + n) / \Gamma(a)$ denotes the $n$-th rising factorial (and its analytic continuation).

\begin{algorithm} \label{alg:fractionalnb}
Let $r>0$ and let $Z_1, Z_2, \dotsc$ be an \iid\ sequence of $p$-coins.
Set $k\defas 1$.
\begin{enumerate}
\item Simulate $W_k \dist \nbdist(\ceiling{r},p)$ with $p$-coins.

\item Compute
$
A(W_k; r) 
	= (r)_{W_k} / (\ceiling{r})_{W_k}
.
$

\item Let $G_k \dist \Uniform$.

\begin{enumerate}
\item If $G_k<A(W_k;r)$, then set $X \defas W_k$;

\item If $G_k \ge A(W_k;r)$, then set $k\defas k+1$ and GO\! TO step 1.
\end{enumerate}
\item Output $X$.
\end{enumerate}
\end{algorithm}

\begin{lem} \label{result:fractionalnb}
\cref{alg:fractionalnb} outputs $X \dist \nbdist(r,p)$, and the expected number of iterations is
$
(1-p)^{\ceiling{r}-r}
	.
$
\end{lem}

\begin{proof}
This is a rejection sampler \cite{robert1999monte} with a proposal distribution $\nbdist(\ceiling{r},p)$, a target distribution $\nbdist(r,p)$, and a constant $k$ chosen such that $k\nbdist(x; \ceiling{r}, p) \ge \nbdist(x; r,p)$, for every $x\in\NNInts$.
Choosing $k = (1-p)^{r-\ceiling{r}}$, the probability that a proposed sample $W \dist \nbdist(\ceiling{r},p)$ is accepted is
\[
A(W; r) = \frac{ \nbdist(W; r,p) }{k \nbdist(W; \ceiling{r},p) }
	= \frac{(r)_{W}}{ (\ceiling{r})_{W} }
	.
\]
The expected number of rejected samples $R$ is geometrically distributed with mean $\frac{1-1/k}{1/k}$, and the expected number of iterations $R+1$ has mean $1/k$.
\end{proof}

In analogy to the work on \emph{Bernoulli factories} \cites{keane1994bernoulli,nacu2005fast,latuszynski2011}, where one wants to simulate $f(p)$-coins (for some function $f$) from $p$-coins when $p$ is unknown, we call \cref{alg:fractionalnb} a \emph{negative binomial factory} and write $$X \given Z_1, Z_2,\dotsc \dist \nbfactory( r,  Z_1, Z_2, \dotsc )$$ to denote that $X$ is simulated from a negative binomial factory with parameter $r$ and input sequence $Z_1, Z_2, \dotsc$.
We generalize \cref{result:intro_arbH} to parameters $r>0$ using random variables from a negative binomial factory in the following algorithm that slightly alters the negative binomial urn scheme.
Let $r>0$ and let $\gprocess{Y_{n,m}}{n,m\in\Nats}$ be an array of simple point processes on $\borelspace$. 
For every $n\in \Nats$,
\begin{enumerate}[leftmargin=*]
\item Define
\[
\Yscr_n 
	\defas \bigcup_{\ell\le \ceiling{r}} \supp(Y_{n,\ell})
	,
\]

\item Put $\Fcal_n \defas \sigma( Y_{n,1}, Y_{n,2}, \dotsc )$, and let $\gprocess{\bar X_{n,j}}{j\in\Nats}$ be a collection of random variables, conditionally independent given $\Fcal_n$, and 
\[
\bar X_{n,j} \given \Fcal_n
	\dist \nbfactory( r, Y_{n,1}\theset{\gamma_{n,j}}, Y_{n,2}\theset{\gamma_{n,j}}, \dotsc )
	,
	\qquad
	j\in \Nats
	.
\]

\item Define a sequence $\gprocess{X}{n\in\Nats}$ of point processes on $\borelspace$ where, for every $j\in\Nats$,
\[
X_n\theset{ \gamma_{n,j} } \defas \bar X_{n,j}
	,
	\quad
	\text{on the event } \{ j\le \kappa_n \}
	,
\]
and $X_n\theset{\gamma_{n,j}} \defas 0$, otherwise.
For every $A\in\bsa$, put $X_n (A) \defas \sum_{s\in A \cap \Yscr_n} X_n \theset s$.
\end{enumerate}

\begin{thm} \label{result:fracNBUS}
When $r>0$, \cref{result:intro_arbH} holds with this construction for $\process X \Nats$.
\end{thm} 

\begin{proof}
Fix $n\in\Nats$.
The proof parallels that for \cref{result:intro_arbH}, except that the construction of the negative binomial variate in \cref{eq:nbfixed} is now given by \cref{alg:fractionalnb} and \cref{result:fractionalnb}.
This verifies the form of the fixed component of $X_n$, however, verifying the form of the ordinary component differs slightly.

Recall the definition of the non-random measure $H_0$ in $\BMspace$ with set of atoms $\Hscr_0$ and non-atomic part $\tilde H_0 \defas H_0 ( \cdot \setminus \Hscr_0 )$.
We must show that the ordinary component $\tilde X_n \defas X_n ( \cdot \setminus \Hscr_0 )$ of $X_n$ is still a Poisson process with intensity measure $r\tilde H_0$.

Recall the definitions of ${\tilde Y_{n,m} \defas Y_{n,m}( \cdot \setminus \Hscr_0 )}$, for every $m\in\Nats$, which are independent Poisson processes on $\bspace$ with intensity measure $\tilde H_0$.
We have that ${\tilde Y_n \defas \sum_{\ell \le \ceiling{r}} \tilde Y_{n,\ell}}$ is a Poisson process on $\bspace$ with intensity measure $\ceiling{r} \tilde H_0$, and by construction $\supp(\tilde X_n) \subseteq \supp(\tilde Y_n)$.
Because ${\cap_{m\in\Nats} \supp(\tilde Y_{n,m}) = \emptyset}$ \as, then for every $s\in \supp(\tilde Y_n)$, the sequence
$
\tilde Y_{n,1} \theset s ,
	\tilde Y_{n,2} \theset s ,
		\dotsc
$
satisfies
\[
\sum_{\ell \le \ceiling{r}} \tilde Y_{n,\ell}\theset s = 1 \quad \as
	,
\]
and
\[
\tilde Y_{n,m} \theset s = 0 \quad \as,
	\qquad \text{for every } m \ge \ceiling{r}+1
	; 
\]
that is, only one entry in the sequence is equal to one \as, which must occur within the first $\ceiling{r}$ entries.
Therefore, independently for every ${s\in \supp(\tilde Y_n)}$, executing step 1 of \cref{alg:fractionalnb} with input sequence $\tilde Y_{n,1}\theset s, \tilde Y_{n,2} \theset s, \dotsc$ will compute
\[
W_1 = 1 \quad \as\ \quad \text{and} \quad W_k = 0 \quad \as,
	\quad \text{for every } k\ge 2
	.
\]
The algorithm therefore outputs $\bar X_{n,j} = 1$ on its first iteration with probability
\[
A(W_1;r) = A(1;r) = \frac{(r)_1}{(\ceiling{r})_1}
	= \frac r {\ceiling{r}}
	,
\]
otherwise it outputs $\bar X_{n,j}=0$ \as\
It follows that $\tilde X_n$ is \as\ simple and, by a Poisson process thinning argument, we have that $\tilde X_n$ is a Poisson process with intensity measure $A(1;r) \ceiling{r} \tilde H_0 = r \tilde H_0$, as desired.
\end{proof}

\appendix

\section{Decompositions of negative binomial distributions}
\label{sec:decomp_nb}

Here we present several basic results on the negative binomial distribution.
The following classic result can be shown using moment generating functions:
\begin{prop}[Sums of negative binomial random variables] \label{result:sum_NB}
Let $(r_i)_{i=1}^n$ be a sequence of positive real numbers and let $p \in (0,1]$.
Let $Z_1, \dotsc, Z_n$ be a collection of independent random variables with
\[
Z_i \dist \nbdist ( r_i, p )
	,
	\quad
	i \le n
	.
\]
Then $\sum_{i=1}^n Z_i \dist \nbdist ( \bar r, p )$, where $\bar r = \sum_{i=1}^n r_i$.  \qed
\end{prop}

For $r\in \Nats$, the negative binomial distribution has an interpretation as describing the number of successes before $r$ failures in a sequence of independent Bernoulli trials, with the probability of success in each trial equal to $p\in (0,1)$:

\begin{lem} \label{result:nb_construct}
Let $r \in \Nats$, let $\process Z \Nats$ be an \iid\ sequence of Bernoulli random variables with success probability $p\in (0,1)$, and let $W_r$ be the random variable in $\NNInts$ given by
\[
W_r \defas \inf
	\Bigl \{
		m\in\NNInts \colon m = \sum_{j=1}^{m+r} Z_j
	\Bigr \}
	.
\]
Then $W_r \dist \nbdist ( r, p )$.
\end{lem}

\begin{proof}
First consider when $r=1$, so that
\[
\Pr \{ W_1 = k \}
	&= \Pr \Bigl \{ \prod_{j=1}^k Z_j = 1 \wedge Z_{k+1} = 0 \Bigr \}
	,
	\qquad k\in\NNInts
	.
\]
Because the $\process Z \Nats$ are \iid, it follows that
\[
\Pr \{ W_1 = k \} = p^k (1-p) ,
	\qquad k\in\NNInts
	,
\]
which is the \pmf\ of the geometric distribution with parameter $p$.
The remainder of the proof follows by induction.  Assume $W_{r-1} \dist \nbdist(r-1, p)$.  Put
\[
W_{r-1}' \defas \inf \Bigl \{
		m\in \NNInts \colon
			m = \sum_{j=W_{r-1}+r}^{m+r} Z_j
	\Bigr \}
	,
\]
so that $W_r = W_{r-1} + W_{r-1}'$ \as\
By the same argument as above, we have that $W_{r-1}' \dist \geodist(p)$, and by \cref{result:sum_NB}, the result follows.
\end{proof}

\section*{Acknowledgements}

We thank Krzysztof {\L}atuszy{\'n}ski for helpful discussions.

\bibliography{nb-process}

\end{document}